\documentclass[12pt]{amsart}

\usepackage{amssymb,amsmath}
\usepackage[dvips]{graphicx}
\newcounter{ENUM}
\newcommand{\be}{\begin{enumerate}}
\newcommand{\ee}{\end{enumerate}}
\newcommand{\beas}{\begin{eqnarray*}}
\newcommand{\eeas}{\end{eqnarray*}}
\newcommand{\bea}{\begin{eqnarray}}
\newcommand{\eea}{\end{eqnarray}}
\newcommand{\beq}{\begin{equation}}
\newcommand{\eeq}{\end{equation}}

\newcommand{\st}{\,:\,}

\newcommand{\ca}{\mathcal{A}}

\newcommand{\cm}{\mathcal{M}}
\newcommand{\cg}{\mathcal{G}}

\newcommand{\fp}{\mathfrak{F}}
\newcommand{\fo}{\mathrm{fo}}
\newcommand{\tq}{\tilde{q}}
\newcommand{\modd}[1]{\,(\mathrm{mod}\, #1)}
\newcommand{\anbc}[2]{\genfrac{\langle}{\rangle}{0pt}{}{#1}{#2}}
\newcommand{\brbc}[2]{\genfrac{[}{]}{0pt}{}{#1}{#2}}

\newtheorem{thm}{Theorem}[section]

\newtheorem{lem}[thm]{Lemma}
\newtheorem{cor}[thm]{Corollary}
\newtheorem{conj}[thm]{Conjecture}

\theoremstyle{definition}

\newtheorem{ex}[thm]{Example}

\theoremstyle{remark}

\numberwithin{equation}{section}

\def\zz{\mathbb{Z}}
\def\nn{\mathbb{N}}
\def\cc{\mathbb{C}}

\newcommand{\qq}{\mathbb{Q}}

\begin{document}
\title{Theorems and Conjectures on Some Rational Generating Functions} 
%\\ (preliminary version)}

\date{\today}

\author{Richard P. Stanley}
\email{rstan@math.mit.edu}
\address{Department of Mathematics, University of Miami, Coral Gables,
FL 33124}

\maketitle

\section{Introduction}
This paper arose from my earlier paper \cite{stern}. (See also the
follow-up by Speyer \cite{speyer}.) The prototypical
result in \cite{stern} is the following. Define
  \beq S_n(x) = \prod_{i=0}^{n-1} \left(1+x^{2^i}+x^{2^{i+1}}\right).
     \label{eq:stpr} \eeq
Set $S_n(x)=\sum_{k\geq 0} \anbc nk x^k$ (a finite sum), with
$S_0(x)=1$, and define
  $$ u_2(n) = \sum_{k\geq 0} \anbc{n}{k}^2. $$
Then
% 1/8/21 changed x^2 to 2x^2
  $$ \sum_{n\geq 0} u_2(n)x^n = \frac{1-2x}{1-5x+2x^2}. $$
Upon seeing this result and some similar ones, Doron Zeilberger asked
what happens when $2^n$ is replaced by some other function satisfying
a linear recurrence with constant coefficients, such as the Fibonacci
numbers $F_n$ (with initial conditions $F_1=F_2=1$). We will prove
some results of this nature, but the data suggests that much more is
true. We give a number of conjectures in this direction.

%This paper is written in somewhat casual style, with many proofs
%omitted or just outlined.
%It is not intended for publication in its
%current form.
%My main motivation in posting it is to inspire someone to find a
%better approach, prove the conjectures, and develop a more general
%theory.

\section{A Fibonacci product} \label{sec:fp}
In this section we consider the product
  \beq I_n(x)=\prod_{i=1}^n \left( 1+x^{F_{i+1}}\right).
  \label{eq:indef} \eeq
In particular, $I_0(x)= 1$ (the empty product) and $I_1(x)=1+x$.
Our main goal for this section is a proof of the following result.

\begin{thm} \label{thm1}
Let $I_n(x)=\sum_{k\geq 0} c_n(k)x^k$, and set
  $$ v_2(n) = \sum_{k\geq 0} c_n(k)^2, $$
so $v_2(0)=1$, $v_2(1)=2$, $v_2(2)=4$, $v_2(3)=10$, etc. Then
  $$ \sum_{n\geq 0} v_2(n)x^n = \frac{1-2x^2}{1-2x-2x^2+2x^3}. $$ 
\end{thm}  

The proof parallels the proofs in \cite{stern} of similar results by
setting up a system of linear recurrences of order one. (In
Section~\ref{sec:genfib} we give another proof, the case $k=2$ and
$t=1$ of Theorem~\ref{thm:vk2n}).  However, deriving these recurrences
here is quite a bit more complicated. It simplifies somewhat the
argument to replace $I_n(x)$ with another power series (with
noninteger exponents). The justification for this replacement is
provided by the following lemma. Part (b) is presumably known, though
we couldn't find this result in the literature.

\begin{lem} \label{lem1}
  Let $\phi=\frac 12(1+\sqrt{5})$.
  \be\item[(a)] Suppose $\alpha=(a_0,a_1,\dots)$ and
  $\beta=(b_0,b_1,\dots)$ are 
  sequences of 0's and 1's, with finitely many 1's, such that
    $$ \sum_{i\geq 0} a_i\phi^i =\sum_{i\geq 0} b_i\phi^i. $$
  Then $\alpha$ can be converted to $\beta$ by a sequence of
  operations that replace three consecutive terms 001 with 110, and
  vice versa. 
   \item[(b)] Suppose $\alpha=(a_0,a_1,\dots)$ and
  $\beta=(b_0,b_1,\dots)$ are 
  sequences of 0's and 1's, with finitely many 1's, such that
    $$ \sum_{i\geq 0} a_iF_{i+2} =\sum_{i\geq 0} b_iF_{i+2}. $$
  Then $\alpha$ can be converted to $\beta$ by a sequence of
  operations that replace three consecutive terms 001 with 110, and
  vice versa. 
  \ee
\end{lem}

\textbf{Proof.} (a) This is a simple consequence of the fact that
$\phi$ is a zero of the irreducible polynomial $x^2-x-1$.

(b) Simple proof by induction on the largest $j$ for which $a_j=1$ or
$b_j=1$. Details omitted. $\ \ \Box$

For a power series $P(x)=\sum_{i\geq 0} c_i x^{m_i}$ with real
exponents $m_i\geq 0$, where each $c_i\neq 0$ and $m_0<m_1<\cdots$, we
call the sequence $(c_0,c_1,\dots)$ the \emph{sequence of
  coefficients} of $P(x)$.  It's easy to see that Lemma~\ref{lem1} has
the following consequence.

\begin{cor}
Let $G_n(x) = \prod_{i=0}^{n-1} \left( 1+x^{\phi^i}\right)$, a
``formal polynomial'' whose exponents lie in the ring
$\zz[\phi]$. Then the sequence of coefficients of $G_n(x)$ is equal to
the sequence of coefficients of $I_n(x)$. Moreover, if the coefficient
of $x^k$ in $I_n(x)$ is 0, then $k>\deg I_n(x)$.
\end{cor}

To illustrate the next result, when we expand $G_5(x)$ we obtain the
following expression, where the terms are listed in increasing order
of their exponents:
  $$ G_5(x) = 1+x+x^a+2x^{a+1}+x^{a+2}+2x^b+2x^{b+1}+x^c+3x^{c+1}+
   2x^{c+2} $$
\vspace{-1em}
  $$ \qquad\qquad +2x^d+3x^{d+1}+x^{d+2} +2x^e+2x^{e+1}+x^f+2x^{f+1}
  +x^{f+2}+x^g+x^{g+1}, $$
for certain numbers $a,b,\dots,g\in\zz[\phi]$. Note that the terms come
in groups (or \emph{strings}) of length two or three, where within
each string the 
exponents increase by one at each step.
% Moreover, it can be checked that this is the only way that two of
% the exponents can differ by an integer.

\begin{thm}
For $n\geq 1$, we can write $G_n(x)$ as a sum $G_n(x) = T_1(x)+ T_2(x)
+\cdots+ T_k(x)$, where each $T_i(x)$ has the form $c_1x^h+c_2x^{h+1}$ or
$c_1x^h+c_2x^{h+1}+c_3x^{h+2}$ for some positive integers
$c_1,c_2,c_3$.. Moreover, the largest exponent 
of a term in $T_i(x)$ is less than the smallest exponent of a term in
$T_{i+1}(x)$. (As an aside, we have $k=F_{n+1}$.)  
\end{thm}  

\textbf{Proof hint.} The terms $T_i(x)$ with two summands are of the
form $1+x$ or
  $$ c_1x^{\phi+\phi^2+a_3\phi^3+a_4\phi^4+\cdots}+
    c_2x^{1+\phi+\phi^2+a_3\phi^3+a_4\phi^4+\cdots}, $$
where $a_3,a_4,\dots$ is a sequence of 0's and 1's with finitely many
1's, and where $c_1,c_2$ are positive integers. Similarly, the terms
$T_i(x)$ with three summands are of the form 
   $$ c_1x^{\phi+a_3\phi^3+a_4\phi^4+\cdots}+
          c_2x^{\phi^2+a_3\phi^3+a_4\phi^4+\cdots}+
         c_3x^{1+\phi^2+a_3\phi^3+a_4\phi^4+\cdots}.\ \ \Box $$
%exponents $\phi+\phi^2,1+\phi+\phi^2,a_3,a_4,\dots$ where
%$a_3,a_4,\dots$ is a sequence of exponents all greater than
%$1+\phi+\phi^2$. Similarly, the terms $T_i(x)$ with three summands
%have exponents $\phi,1+\phi,1+\phi^2,b_4,b_5,\dots$ where
%$b_4,b_5,\dots$ is a sequence of exponents all greater than
%$1+\phi^2$. $\ \ \Box$

\textsc{Note.} Though we have no need of this result, let us mention
that if $d_n(i)$ denotes the number of terms (either two or three) of
$T_i(x)$ (coming from $G_n(x)$), then $d_n(i)=d_n(F_{n+1}-i+1)$ and 
  $$ d_n(i) = 1+\lfloor i\phi\rfloor -\lfloor (i-1)\phi\rfloor,\ \
  1\leq i \leq \left\lceil \frac 12F_{n+1}\right\rceil. $$
Set $d(i)=\lim_{n\to\infty}d_n(i)$. The sequence $(d(1),d(2),\dots)$
is obtained from sequence A014675 in OEIS by prepending a 1 and adding
1 to every term. 

We now define an array analogous to Pascal's triangle (or the
arithmetic triangle) and Stern' triangle of \cite{stern}. We call the
resulting array the \emph{Fibonacci triangle} $\mathcal{F}$. (This
definition is unrelated to some other definitions of Fibonacci
triangle in the literature.)

Every row is a sequence of positive integers, together with a grouping
of consecutive terms such that every string of the grouping has two or
three terms. We will denote the grouping by a bullet ($\bullet$)
between stringss. The first row is the sequence $1,1$, which necessarily
has a single element $1,1$ in its grouping.  Regard the first entry in
each row as preceded by a 0.  Similarly, the last entry in each row is
followed by a 0. 

Row $i+1$ is obtained from row $i$ by the following recursive
procedure. If a term $a_j$ of row $i$ ends a string (so $a_{j+1}$
begins a string), then below $a_j,a_{j+1}$ write in row $i+1$ the
3-element string $a_j,a_j+a_{j+1},a_{j+1}$. If $a_j$ in row $i$ is the
middle element of a 3-element string, then write in row $i+1$ below
$a_j$ the 2-element string $a_j,a_j$.

Note that according to this procedure, odd numbered rows will begin with
a 2-element string $1,1$ (preceded by a 0) and end with a 2-element
string $1,1$ (followed by a 0). On the other hand, even numbered rows
will begin with the 3-element string $0,1,1$ and end with the 3-element
string $1,1,0$. In all instances, entries equal to 0 are not regarded
as actual entries of $\mathcal{F}$.

The first five rows of $\mathcal{F}$ look as follows:

{\small
  $$\hspace{-2em} \begin{array}{ccccccccccccccccccccccccccccc}
    & & & & & & & & & & & 1 & & & & & & 1\\
    & & & & & 1 & & & & & & 1 & & & \bullet & & & 1 & & & & & & 1\\
    & & & 1 & & 1 & & & \bullet & & & 1 & & & 2 & & & 1 & & & \bullet
      & & & 1 & & 1\\
    & 1 & & 1 & \bullet & 1 & & & 2 & & & 1 & \bullet & 2 & & 2 &
      \bullet & 1 & & & 2 & & & 1 & \bullet & 1 & & 1\\
  1 & 1 & \bullet & 1 & 2 & 1 & \bullet & 2 & & 2 & \bullet & 1 & 3 &
     2 & 
    \bullet & 2 & 3 & 1 & \bullet & 2 & & 2 & \bullet & 1 & 2 & 1
    & \bullet & 1 & 1\\
  \end{array} $$
}
  
Let $\brbc nk$ be the $k$th entry (beginning with $k=0$) in row
$n$ (beginning with $n=1$) of $\mathcal{F}$. Set $H_n(x) =
\sum_{k\geq 0} \brbc nk x^k$. For instance,
  $$ H_3(x) = 1+x+x^2+2x^3+x^4+x^5+x^6. $$
The following result can be proved by induction.

\begin{thm} \label{thm:hnfn}
We have $H_n(x)=I_n(x)$.
\end{thm}

We now have all the ingredients for proving Theorem~\ref{thm1}. 
Define
   \beas \brbc nk_1 & = & \left\{ \begin{array}{rl}
       \brbc nk, & \mbox{if the $k$th entry in row $n$
       of $\mathcal{F}$ is the first entry of its
       string}\\[.5em] 0, & \mbox{otherwise} \end{array}
       \right.\\
    \brbc nk_2 & = & \left\{ \begin{array}{rl}
       \brbc nk, & \mbox{if the $k$th entry in row $n$
       of $\mathcal{F}$ is the middle entry of its
       string}\\[.5em] 0, & \mbox{otherwise} \end{array}
       \right.\\
    \brbc nk_3 & = & \left\{ \begin{array}{rl}
       \brbc nk, & \mbox{if the $k$th entry in row $n$
       of $\mathcal{F}$ is the last entry of its
       string}\\[.5em] 0, & \mbox{otherwise.} \end{array}
       \right. \eeas
Set
 \beas A_1(n) & = & \sum_k \brbc nk_1^2\\       
      A_2(n) & = & \sum_k \brbc nk_2^2\\       
      A_3(n) & = & \sum_k \brbc nk_3^2\\       
      A_{3,1}(n) & = & \sum_k \brbc nk_3 \brbc{n}{k+1}_1\\       
      A_{1,2}(n) & = & \sum_k \brbc nk_1 \brbc{n}{k+1}_2\\       
      A_{1,3}(n) & = & \sum_k \brbc nk_1 \brbc{n}{k+1}_3\\       
      A_{2,3}(n) & = & \sum_k \brbc nk_2 \brbc{n}{k+1}_3. \eeas
Using the definition of $\mathcal{F}$ one checks the following (all
sums are over $k\geq 0$):
   \beas A_1(n+1) & = & \sum \left( \brbc nk_3^2 + \brbc
   nk_2^2\right)\\ & = & A_2(n) + A_3(n)\\[.5em]  
   A_2(n+1) & = & \sum \left( \brbc nk_3 + \brbc{n}{k+1}_1\right)^2\\
   & = & A_1(n) + A_3(n)+2A_{3,1}(n)\\[.5em]
   A_3(n+1) & = & \sum \left( \brbc nk_1^2 + \brbc
   nk_2^2\right)\\ & = & A_1(n) + A_2(n)\\[.5em]
   A_{3,1}(n+1) & = & \sum \left(\brbc nk_1\brbc{n}{k+1}_2+
   \brbc nk_1\brbc{n}{k+1}_3+\brbc nk_2\brbc{n}{k+1}_3\right)\\ & = &
    A_{1,2}(n)+A_{1,3}(n)+A_{2,3}(n) \eeas

  \beas  A_{1,2}(n+1) & = & \sum \brbc nk_3\left( \brbc{n}{k}_3+
    \brbc{n}{k+1}_1\right)\\ & = & A_3(n)+A_{3,1}(n)\\[.5em]
   A_{1,3}(n+1) & = & \sum \brbc nk_2^2\\ & = & A_2(n)\\[.5em]
   A_{2,3}(n+1) & = & \sum \left( \brbc nk_3+\brbc{n}{k+1}_1\right)
   \brbc{n}{k+1}_1\\ & = & A_1(n) + A_{3,1}(n). \eeas
Let $M$ denote the matrix
 $$ M=\left[ \begin{array}{ccccccc} 0 & 1 & 1 & 0 & 0 & 0 & 0\\
    1 & 0 & 1 & 2 & 0 & 0 & 0\\ 1 & 1 & 0 & 0 & 0 & 0 & 0\\
    0 & 0 & 0 & 0 & 1 & 1 & 1\\ 0 & 0 & 1 & 1 & 0 & 0 & 0\\
    0 & 1 & 0 & 0 & 0 & 0 & 0\\ 1 & 0 & 0 & 1 & 0 & 0 & 0
    \end{array}\right], $$
and let $v(n)$ denote the column vector
  $$ v(n)=[A_1(n),A_2(n),A_3(n),
  A_{3,1}(n), A_{1,2}(n), A_{1,3}(n), A_{2,3}(n)]^t $$
(where $^t$ denotes transpose). The recurrences above take the form
${v(n+1) = Mv(n)}$. Hence, as in \cite[{\S}2]{stern}, the seven functions
$A_\alpha(n)$ all satisfy a linear recurrence relation whose characteristic
polynomial $Q_2(x)$ is the characteristic polynomial $\det(xI-M)$ of
$M$. Then $\sum_{n\geq 0} A_\alpha(n)x^n$ is a rational function with
denominator $x^{\deg Q_2(x)}Q_2(1/x)$.  One computes
$Q_2(x)=x^2(x+1)^2(x^3-2x^2-2x+2)$. Taking into account the initial
conditions for the case $A_1(n)+A_2(n)+A_3(n)=v_2(n)$ yields
Theorem~\ref{thm1}. 

Note that the factors $x^2(x+1)^2$ of $Q_2(x)$ were spurious. This
suggests that there should be a simpler argument involving a $3\times
3$ matrix rather than a $7\times 7$ matrix.  

\section{A Fibonacci triangle poset}
This section assumes a basic knowledge of the combinatorics of
partially ordered sets (posets) and symmetric functions such as that
appearing in \cite[Ch.~3]{ec1} and \cite[Ch.~3]{ec2}. It is unrelated
to the rest of this paper.

There is a poset $\fp$ that is naturally associated with the Fibonacci
triangle $\mathcal{F}$. The elements $t_{nk}$ of $\fp$ correspond to
the entries $\brbc{n}{k}$ of $\mathcal{F}$, with a bottom element
$\hat{0}$ adjoined. The element $t_{n+1,k}$ covers $t_{nj}$ if the
recurrence defining $t_{n+1,k}$ involves $t_{nj}$. Thus every element
of $\fp$ is covered by exactly two elements. The number of saturated
chains from $\hat{0}$ to $t_{nk}$ is $\brbc nk$. See
Figure~\ref{fig:fp}, which is drawn upside-down (as a poset) in order
to agree with the way the Fibonacci triangle $\mathcal{F}$ is
drawn. We call $\fp$ the \emph{Fibonacci triangle poset} or
\emph{FT-poset}. For another representation of this poset (considered
as a ``hyperbolic graph''), see Northshield
\cite[Fig.~4]{north}. (There is already a poset called the
\emph{Fibonacci poset} \cite{rs:fp}. For a further poset associated
with Fibonacci numbers see \cite[{\S}5]{rs:dp}.)

  \begin{figure}
    \centering
  \centerline{\includegraphics[width=12cm]{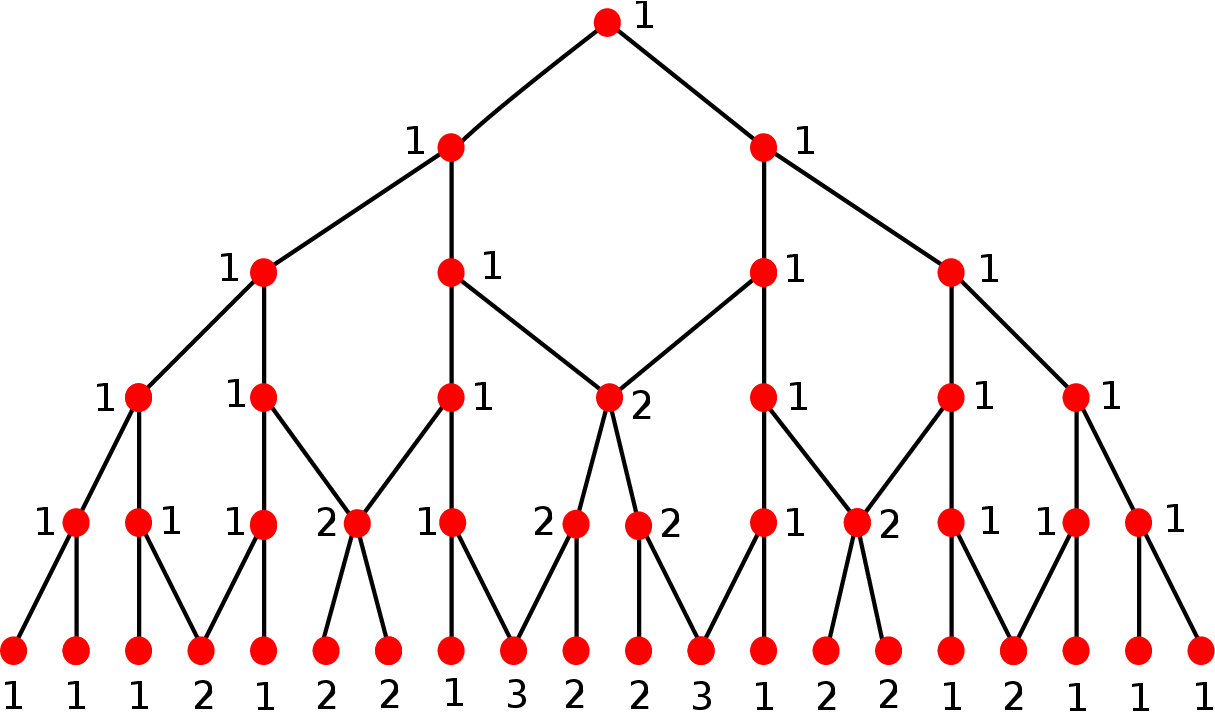}}
  \caption{The Fibonacci triangle poset $\fp$}
  \label{fig:fp}
  \end{figure}

The grouping of the elements of row $n$ of the Fibonacci triangle into
strings of size two and three is readily seen from the
FT-poset. Consider the subposet consisting of ranks $n-1$ and $n$
(where the bottom element $\hat{0}$ has rank 0). The connected
components of this subposet define the grouping.

Suppose that we label the edges of $\fp$ as follows. The edges between
ranks $2k$ and $2k+1$ are labelled alternately
$0,F_{2k+2},0,F_{2k+2},\dots$ from left to right. The edges between
ranks $2k-1$ and $2k$ are labelled alternately
$F_{2k+1},0,F_{2k+1},0,\dots$ from left to right. See
Figure~\ref{fig:fiblabel}. Then it is not difficult to show that if
$t\in \fp$ and rank$(t)=n$, then the edge labels of all
saturated chains from $\hat{0}$ to $t$ have the same sum $\sigma(t)$,
and that these chains correspond to all ways to write $\sigma(t)$ as a
sum of the elements of a subset of $\{F_2,F_3,\dots,F_{n+1}\}$. 

  \begin{figure}
    \centering
    \centerline{\includegraphics[width=10cm]{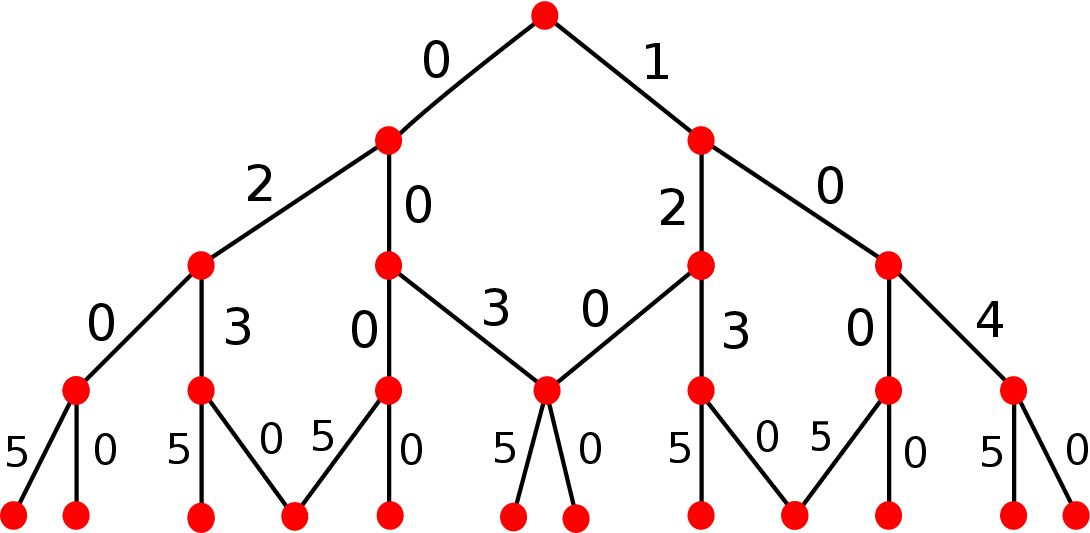}}
% 1/7/21 next line changed "labelling"    
  \caption{An edge labeling of $\fp$}
  \label{fig:fiblabel}
  \end{figure}

Now suppose that we label every point $t\in \fp$ by $\sigma(t)$. Thus
the labels at rank $n$ consist of all integers $0,1,\dots, F_{n+3}-2$,
since these are the exponents when we expand $I_n(x)$. The sequence
$S(n)$ of labels at rank $n$, read left-to-right, is a subsequence of
the sequence of labels at rank $n+1$, read left-to-right. Thus the
sequences $S(1),S(2),S(3),\dots$ approach a limit, which is a dense
linear order on the nonnegative integers that we denote by
$\prec$. For instance, from Figure~\ref{fig:fiblabel} we see that
$S(4)=(7,2,10,5,0,8,3,11,6,1,9,4)$.

The order $\prec$ can be described as
follows. Let $0\leq m<n$. Every nonnegative integer has a unique
representation as a sum of nonconsecutive Fibonacci numbers, where a
summand equal to 1 is always taken to be $F_2$ (Zeckendorf's theorem). Let
$m=F_{i_1}+\cdots+F_{i_r}$ and $n=F_{j_1}+\cdots+F_{j_s}$ be such
representations, with $i_1<\cdots<i_r$ and $j_1<\cdots<j_s$. Regard
$F_{i_{r+1}}=F_{j_{s+1}}= F_0=0$. Let $k$ be the least index for which
$i_k\neq j_k$. Then we have $m\prec n$ precisely in the following
cases.
  \begin{itemize}
  \item $i_k$ and $j_k$ are odd, and $i_k<j_k$.
  \item $i_k$ is odd and $j_k$ is even.
  \item $i_k$ is even and nonzero, $j_k$ is even and nonzero, and
    $i_k>j_k$.
  \item $i_k=0$ and $j_k$ is even.
  \end{itemize}
  For instance, let $n\neq 0$. Then $n\prec 0$ if $j_1$ is odd, while
  $n\succ 0$ if $j_1$ is even.
  
The poset $\fp$ is not ``nice'' in regard to its topological
properties. For instance, the rank-selected subposets $\fp_{n-1,n}$,
$n\geq 2$, are not connected, so $\fp$ is not
Cohen-Macaulay. Moreover, its flag $h$-vector $\beta_\fp$
\cite[{\S}3.13]{ec1} can be negative, e.g.,
$\beta_\fp(1,2)=-1$. Despite these shortcomings, $\fp$ does have some
nice structural and enumerative properties which we now discuss.

A poset $P$ is called \emph{upper homogeneous} or \emph{upho}
\cite{upho}\cite{rs:upho} if for every $t\in P$, the dual principal
order ideal $V_t=\{s\in P\st s\geq t\}$ is isomorphic to $P$. It is
easily seen that $\fp$ is upho. In fact, $\fp$ has an especially
simple structure. For $i,b\geq 2$ define the upho poset $P_{ib}$ by
the following conditions:
\begin{itemize}
  \item $P$ has a unique minimal element $\hat{0}$.
  \item Every element of $P_{ib}$ is covered by $i$ elements.
  \item $P_{ib}$ has a planar (i.e., no crossing edges) Hasse diagram
    such that if $u,u'$ are 
    consecutive (reading the Hasse diagram from left-to-right) covers
    of $t$, then the elements $t,u,u'$ ``extend to a $2b$-gon.''
    That is, there is an element $v>t$ for which the Hasse
    diagram of the interval $[t,v]$ contains $u,u'$ and looks like a
    $2b$-gon with no vertices or edges in its interior, and where $t$
    and $v$ are antipodal edges (so the interval $[t,v]$ is graded).
    (See \cite[{\S}4]{upho} for more information on planar upho posets.)
\end{itemize}
It's not hard to see that $P_{ib}$ exists and is unique up to
isomorphism. In particular $\fp\cong P_{23}$, a surprisingly simple
description of $\fp$.  Moreover, the poset
corresponding to Pascal's triangle (i.e., the product of two chains
$t_0<t_1<\cdots$) is isomorphic to $P_{22}$, while the poset
corresponding to Stern's triangle \cite[p.~25]{rs:upho}\cite{yang} is
isomorphic to $P_{32}$. 

Recall now that if $P$ is a finite
graded poset with $\hat{0}$ and $\hat{1}$, then the \emph{Ehrenborg
  quasisymmetric function} $E_P$ of $P$
\cite[Def.~4.1]{ehren}\cite[Exer.~7.48]{ec2} is defined by
    $$ E_P = \sum_{\hat{0}=t_0\leq t_1\leq \cdots \leq
 t_{k-1}<t_k=\hat{1}} x_1^{\rho(t_0,t_1)}x_2^{\rho(t_1,t_2)}\cdots
 x_k^{\rho(t_{k-1},t_k)}, $$
where $\rho(s,t)$ denotes the rank (length) of the interval
$[s,t]$. (The sum ranges over all multichains from $\hat{0}$ to
$\hat{1}$ of all possible
lengths $k\geq 1$ such that $\hat{1}$ occurs with multiplicity one.)
$E_P$ is a kind of generating function for the flag $h$-vector
$\beta_P$ of $P$---knowing $E_P$ is equivalent to knowing $\beta_P$.
If $P$ is infinite, graded, with $\hat{0}$ and with finitely many
elements $q_n$ of each rank $n\geq 0$, then we can extend the
definition of $E_P$ by setting
  $$ E_P = \sum_{t\in P} E_{[\hat{0},t]}. $$

One nice feature of graded upho posets $P$ with finitely many elements
of every rank is that $E_P$ is a symmetric function. If $P$ has 
$q_n$ elements of rank $n$ (so $q_0=1$), then in fact
 $$ E_P = \sum_\lambda q_{\lambda_1}q_{\lambda_2}\cdots m_\lambda, $$
where $\lambda$ ranges over all partitions
$(\lambda_1,\lambda_2,\dots)$ of all nonnegative integers, and where
$m_\lambda$ is a monomial symmetric function. Equivalently, define
  $$ \Phi(P,x) = \sum_{n\geq 0} q_nx^n, $$
the \emph{rank-generating function} of $P$. (The usual notation is
$F(P,q)$, but that might cause confusion with $F_n^{(k)}$ as defined
in {\S}5.) Then \cite[Lemma~2.3]{upho} 
  \beq E_P =
  \Phi(P,x_1)\Phi(P,x_2)\Phi(P,x_3)\cdots. \label{eq:epprod} \eeq 

Let us also note that if $1\leq r_1<r_2<\cdots<r_k$, then the flag
$f$-vector $\alpha_P$ of $P$ at $S=\{r_1,r_2,\dots,r_k\}$ is given by
  $$ \alpha_P(S) = q_{r_1}q_{r_2-r_1}q_{r_3-r_2}\cdots q_{r_k-r_{k-1}},
  $$
since there are $q_{r_1}$ ways to choose an element $t_1$ of rank
$r_1$, then $q_{r_2-r_1}$ ways to choose an element $t_2$ of rank
$r_2$ satisfying $t_2>t_1$, etc.
  
It's not hard to see (since every element of $P_{ib}$ is covered by
$i$ elements, and every element of rank $n-b$ is the bottom element of
$i-1$ $2b$-gons whose top element has rank $n$) that 
  \beq \Phi(P_{ib},x) = \frac{1}{1-ix+(i-1)x^b}. \label{eq:rgf} \eeq
  Equivalently, $q_n$ satisfies the initial conditions and recurrence
  $$ q_0=1, q_1=i, q_2=i^2,\dots,q_{b-1}=i^{b-1} $$
  $$ q_n=iq_{n-1}-(i-1)q_{n-b},\ \ n\geq b. $$
  
In particular, for $\fp\cong P_{23}$ we have
   \beas \Phi(\fp,x) & = & \frac{1}{1-2x+x^3}\\
       & = & \frac{1}{(1-x)(1-x-x^2)}\\ & = &
     \sum_{n\geq 0}(F_{n+3}-1)x^n. \eeas
In other words, the Fibonacci triangle $\mathcal{F}$ has $F_{n+3}-1$
elements in row $n$. Of course this is easy to see by a more direct
argument.

The symmetric functions $E_{P_{ib}}$ have ``nice'' expansions in terms
of the power sum symmetric functions $p_\lambda$ and the forgotten
symmetric functions $\fo_\lambda=\omega m_\lambda$, where $\omega$ is
the standard involution on symmetric functions. Define $\tilde{q}_n$
by
  $$ \tq_0=1, \tq_1=i,\ \tq_2=i^2, \dots,\ \tq_{b-1}=i^{b-1},
   \tq_b=i^b-b(i-1) $$
  $$ \tq_n = i\tq_{n-1}-(i-1)\tq_{n-b},\ n\geq b+1. $$
Thus $\tq_n$ satisfies the same recurrence as $q_n$. but beginning at
$n=b+1$, not $n=b$, and with different initial conditions. We use
notation such as $b^31^4$ to denote the partition $(b,b,b,1,1,1,1)$

\begin{thm}
  \be \item[(a)] We have
   \beq E_{P_{ib}} = \sum_\lambda z_\lambda^{-1} \tq_{\lambda_1}
   \tq_{\lambda_2}\cdots\,p_\lambda,
     \label{eq:epij} \eeq
where $\lambda$ ranges over all partitions
$(\lambda_1,\lambda_2,\dots)$ of all $n\geq 0$.
  \item[(b)] We have
    $$ E_{P_{ib}} = \sum_{n\geq 0} \sum_{j\geq 0}
    (-1)^{jb}(i-1)^j i^{n-jb}\fo_{b^j1^{n-jb}}, $$
where we set $\fo_{b^j1^{n-jb}}=0$ if $jb>n$.
   \ee
\end{thm}

\begin{proof}
\be\item[(a)] By equations~\eqref{eq:epprod} and \eqref{eq:rgf} we
have
  \beq E_{P_{ib}} = \frac{1}{\prod_m(1-ix_m+(i-1)x_m^b)}. \label{eq:eppr}
  \eeq 
Let $1-ix+(i-1)x^b=\prod_{h=1}^b(1-\alpha_h x)$, where
$\alpha_h\in\cc$. Then 
   \bea \log E_{P_{ib}} & = & -\sum_m \sum_h \log(1-\alpha_hx_m)
      \nonumber\\
      & = & \sum_m \sum_h \sum_{k\geq 1} \alpha_h^k\frac{x_m^k}{k}
      \nonumber\\
  & = & \sum_{k\geq 1}\left(\sum_h \alpha_h^k\right)\frac{p_k}{k},
   \label{eq:alphak} \eea
where $p_k=\sum_m x_m^k$, the $k$th power sum symmetric function of
the $x_m$'s. Now for any polynomial $Q(x)$ with $Q(0)=1$, say
$Q(x)=\prod(1-\beta_hx)$,  we have
  $$ \sum_{k\geq 1}\left(\sum_h \beta_h^k\right)x^k =
      \frac{-xQ'(x)}{Q(x)}. $$ 
Letting $Q(x)=1-ix+(i-1)x^b$ gives
  \beas \sum_{k\geq 1} \left(\sum_h \alpha_h^k\right)x^k & = &
  \frac{ix-(i-1)bx^b}{1-ix+(i-1)x^b}\\ & = &
  ix+i^2x^2+\cdots+i^{b-1}x^{b-1}\\ & &
  \qquad +(i^b-(i-1)b)x^b+\cdots. \eeas
It follows easily that $\sum_h \alpha_h^k = \tq_k$.

Now apply the exponential function exp to \eqref{eq:alphak}. By e.g.\
\cite[Prop.~7.7.4]{ec2}, we obtain equation~\eqref{eq:epij}.    
 \item[(b)] Let $R(x)$ be any power series with constant term 1, and
let $\Gamma=R(x_1)R(x_2)\cdots$. If $\omega$ is the standard
involution on symmetric functions, then an elementary argument gives
  $$ \omega\Gamma = \frac{1}{R(-x_1)R(-x_2)\cdots}. $$
Hence by equation~\eqref{eq:eppr},
  \beas \omega E_{P_{ib}} & = &  \prod_m
  \left(1+ix_m+(-1)^b(i-1)x_m^b\right)\\ & = &
  \sum_{n\geq 0}\sum_{j\geq 0} (-1)^{jb}(i-1)^ji^{n-jb}m_{b^j1^{n-jb}}.
%  \left( i^nm_{1^n}+
%   (-1)^b(i-1)i^{n-b}m_{b1^{n-3}}
%    +(i-1)^2i^{n-2b}m_{b^21^{n-2b}}\right.\\
%    & & \left.\qquad +(-1)^b(i-1)^3i^{n-3b}m_{b^31^{n-3b}}+\cdots\right).
    \eeas
Since $\omega$ interchanges $m_\lambda$ and $\fo_\lambda$, the proof follows.
\ee  
\end{proof}

\textsc{Note.} Fix $i$ and $b$. Let $t_{nk}$ be the $k$th element from
the left in the $n$th row, beginning with $k=0$, of $P_{ib}$. Write
$\brbc nk$ for the number of saturated chains from $\hat{0}$ to
$t_{nk}$, and as usual let $q_n$ be the number of elements of $P_{ib}$
of rank $n$. It is immediate from the recurrence
$q_n=iq_{n-1}-(i-1)q_{n-b}$ and the initial conditions for
$q_0,\dots,q_{i-1}$ that $q_n-q_{n-1}$ is divisible by $i-1$. Set
$r_n=(q_n-q_{n-1})/(i-1)$. Then it can be shown that
  $$ \sum_k \brbc nk x^k =\prod_{j=1}^n \left(1+x^{r_j}+x^{2r_j}+
    \cdots + x^{(i-1)r_j}\right). $$
It might be interesting to further investigate the posets
$P_{ib}$. For the case $P_{32}$ (Stern's poset), see Yang \cite{yang}.  

\section{Some generalizations}
There are several ways we can try to generalize Theorem~\ref{thm1}.
In this section we will consider generalizing the product $I_n(x)$ and
the function $v_2(n)$. However, we continue to deal with Fibonacci
numbers. Let $\alpha=(\alpha_0,\alpha_1,\dots,\alpha_{m-1})\in \nn^m$
(where $\nn=\{0,1,\dots\}$), and define
  $$ v_\alpha(n) = \sum_{k\geq 0}\brbc nk^{\alpha_0}
   \brbc{n}{k+1}^{\alpha_1}\cdots\brbc{n}{k+m-1}^{\alpha_{m-1}}. $$
This definition is completely analogous to the definition of
$u_\alpha(n)$ in \cite{stern}. As in \cite{stern}, we write
$v_{\alpha_0,\dots,\alpha_{m-1}}$ as short for
$v_{(\alpha_0,\dots,\alpha_{m-1})}$.

Our proof of Theorem~\ref{thm1} carries over to the following
result. The argument is analogous. We just have to ascertain that we
don't end up with a system of infinitely many equations. This is
proved in the same way as in \cite[Thm.~2]{stern}.

\begin{thm}
  For any $\alpha\in\nn^m$, the generating function
  $$ J_\alpha(x) = \sum_{n\geq 0} v_\alpha(n)x^n $$
is rational.
\end{thm}

We used the Maple package gfun to ``guess'' the rational function
$J_\alpha(x)$ for some small $\alpha$. Gfun finds the ``simplest''
rational function fitting the data, which consists of values of
$v_\alpha(n)$ for small $n$ (typically around $0\leq n\leq
36$). Subsequently Zeilberger \cite{zeil} developed a Maple package
\texttt{SternCF.txt} that can make such computations using variants of
the linear algebra method of Section~\ref{sec:fp}. Thus he obtains
rigorous proofs of results like the following examples, where
$\alpha=(r)$. No guesswork using gfun is necessary.

\beq\label{eq:data} \begin{aligned}
 J_3(x) &=\frac{1-4x^2}{1-2x-4x^2+2x^3}\\
  J_4(x) &=\frac{1-7x^2-2x^4}{1-2x-7x^2-2x^4+2x^5}\\
  J_5(x) &=\frac{1-11x^2-20x^4}{1-2x-11x^2-8x^3-20x^4+10x^5}\\
  J_6(x) &=
  \frac{1-17x^2-88x^4-4x^6}{1-2x-17x^2-28x^3-88x^4+26x^5-4x^6+4x^7}\\
  J_7(x) &=\frac{1-26x^2-311x^4-84x^6}
  {1-2x-26x^2-74x^3-311x^4+34x^5-84x^6+42x^7}.
   \end{aligned} \eeq
Note that the denominators all have odd degree, and the numerator is
the even part of the denominator. This behavior has been verified
empirically (not rigorously) for $n\leq 17$. For $8\leq n\leq 17$, the
denominator degrees are $9,7,9,9,13,11,13,11,13,13$, respectively. See
Conjecture~\ref{conj:drx} for a generalization.

Here are some examples where $\alpha$ has at least two terms:
\beas J_{1,1}(x) & = & \frac{x+x^2}{1-2x-2x^2+2x^3}\\[.5em]
  J_{1,0,1}(x) & = & \frac{2x^2+x^3-x^4}{(1-x)(1-2x-2x^2+2x^3)}\\[.5em]
  J_{2,1}(x) & = & \frac{x+x^2}{1-2x-4x^2+2x^3}\\[.5em]
  J_{1,3}(x) & = & \frac{x+x^2+x^3+x^4}{1-2x-7x^2-2x^4+2x^5}\\[.5em]
  J_{2,2}(x) & = & \frac{x+x^2-x^3-x^4}{1-2x-7x^2-2x^4+2x^5}\\[.5em]
  J_{2,3}(x) & = & \frac{x+x^2-x^3-x^4}{1-2x-11x^2-8x^3-20x^4+10x^5}\\[.5em]
  J_{1,1,1}(x) & = & \frac{2x^2+2x^3-2x^4}{(1-x)(1-2x-4x^2+2x^3)}\\[.5em]
  J_{1,0,2}(x) & = & \frac{2x^2+x^3-2x^4+x^5}{(1-x)^2(1-2x-4x^2+2x^3)}\\[.5em]
  J_{2,1,1}(x) & = &
    \frac{2x^2+2x^3-4x^4+4x^5}{(1-x)^2(1-2x-7x^2-2x^4+2x^5)}\\[.5em]
  J_{1,2,1}(x) & = & \frac{2x^2+4x^3-2x^4}{(1-x)(1-2x-7x^2-2x^4+2x^5)} 
  \eeas
It appears that $J_\alpha(x)$ has a denominator of the form
$(1-x)^{c_\alpha}D_r(x)$, where $c_\alpha\geq 0$, $r=\sum \alpha_i$,
and $D_r(x)$ is the denominator of $J_r(x)$. This heuristic
observation is in complete analogy to \cite[Thm.~3]{stern} and
presumably has a similar proof.

We can also generalize the definition of $I_n(x)$. In analogy to
\cite[Thm.~4]{stern} we have the following conjecture.

\begin{conj} \label{conj1}
Let $h\geq 1$, $(a_1,\dots,a_h)\in\cc^h$, and $P(x)\in\cc[x]$. Set 
  $$ I_{h,P,n}(x) = P(x)\prod_{i=1}^n\left(1+a_1x^{F_i}+
    a_2x^{F_{i+1}}+\cdots+a_hx^{F_{i+h-1}}\right). $$
Regarding $h,P$ as fixed, let $c_n(p)$ denote the coefficient
of $x^p$ in $I_{h,P,n}(x)$. For $\alpha=(\alpha_0,\dots,
\alpha_{m-1})\in\nn^m$ define
  $$ v_{h,P,\alpha}(n) = \sum_{p\geq 0}c_n(p)^{\alpha_0}
   c_n(p+1)^{\alpha_1}\cdots c_n(p+m-1)^{\alpha_{m-1}}. $$
Then the generating function $\sum_{n\geq 0}v_{h,P,\alpha}(n)x^n$ is
rational. 
\end{conj}

Let us consider one simple special case of this conjecture. Let $t$ be
any complex number (or an indeterminate), and define
   $$ I_{n,t}(x) = \prod_{i=1}^n\left( 1+tx^{F_{i+1}}\right). $$
We now get a triangle $\mathcal{F}(t)$ with the same grouping into
strings of length two
or three as in $\mathcal{F}$, but the first row is $1,t$. Row
$i+1$ is obtained from row $i$ by the following recursive
procedure. If a term $a_j$ of row $i$ ends a string (so $a_{j+1}$
begins a string), then below $a_j,a_{j+1}$ write in row $i+1$ the
3-element string $a_j,ta_j+a_{j+1},ta_{j+1}$. If $a_j$ in row $i$ is the
middle element of a 3-element string, then write in row $i+1$ below
$a_j$ the 2-element string $a_j,ta_j$.

The following result now is proved in complete analogy with the proof
of Theorem~\ref{thm1}.

\begin{thm} \label{thm1t}
Let $v_{2,t}(n)$ denote the sum of the squares of the coefficients of
$I_{n,t}(x)$. Then
 $$ \sum_{n\geq 0} v_{2,t}(n)x^n =
  \frac{1-t(t^2+1)x^2}{1-(t^2+1)x-t(t^2+1)x^2+t(t^4+1)x^3}. $$
\end{thm}

The polynomial $I_{n,-1}(x)=\prod_{i=1}^n\left(1-x^{F_{i+1}}\right)$
has been considered before. It was shown by Yufei Zhao \cite{zhao}
that all its nonzero coefficients are equal to $\pm 1$. Thus
$v_{2,-1}(n)$ is equal to the number of nonzero coefficents of
$I_{n,-1}(x)$, with generating function
  $$ \sum_{n\geq 0}v_{2,-1}(n)x^n = \frac{1+2x^2}{1-2x+2x^2-2x^3}. $$
This fact is stated (in equivalent form) in the OEIS \cite{oeis}.
Note that we can also directly compute, using the technique in the
proof of Theorem~\ref{thm1}, that $v_{4,-1}=v_{2,-1}$. This gives a
new proof (albeit involving a cumbersome computation) of Zhao's
result.   

\begin{ex}
As a somewhat random special case of Conjecture~\ref{conj1}, let
$h=3$, $(a_1,a_2,a_3)=(0,1,1)$, $P(x)=1$, and $\alpha=(2)$. Thus we
are considering the sum $w(n)$ of the squares of the coefficients of
the product
$\prod_{i=1}^n\left( 1+x^{F_{i+1}}+x^{F_{i+2}}\right)$. Then gfun
suggests and Zeilberger \cite[p.~15]{zeil} confirms that
 {\small  $$ \sum_{n\geq 0}
   w(n)x^n=\frac{1-4x-5x^2+24x^3+4x^4-34x^5+2x^6+10x^7-4x^8}
   {1-7x+x^2+47x^3-32x^4-84x^5+50x^6+34x^7-18x^8}. $$}
In fact, Zeilberger is able to compute the generating function for
$\alpha=(r)$ when $1\leq r\leq 6$. For $r=6$ the denominator has
degree 405.
\end{ex} 

\textsc{Note.} There is an alternative way of describing the nonzero
coefficients of the polynomial $I_n(x)=\prod_{i=1}^n \left(
1+x^{F_{i+1}}\right)$. Let $\ca_n$ denote the set of all words of
length $n$ in the letters $a,b$, so $\#\ca_n=2^n$. Define
$\pi,\sigma\in\ca_n$ to be \emph{equivalent} if $\sigma$ can be
obtained from $\pi$ by a sequence of substitutions (on three
consecutive terms) $baa\to abb$ and $abb\to baa$, an obvious
equivalence relation $\sim$. For instance, when $n=5$ one of the
equivalence classes is $\{baaaa,abbaa,ababb\}$. The quotient monoid of
the free monoid generated by $a,b$ modulo $\sim$ is called the
\emph{Fibonacci monoid} in \cite{wikifm}, though other monoids are
also called the Fibonacci monoid. Here we are interested not in the
monoid itself, but rather the sizes of its equivalence classes. It
follows easily from Lemma~\ref{lem1}(b) that the multiset $M_n$ of
equivalence class sizes of $\sim$ on $\ca_n$ coincides with the
multiset of (nonzero) coefficients of $I_n(x)$. Thus if $u^*_n(r) =
\sum_{j\in M_n} j^r$ ($r\in\nn$), then the generating function
$\sum_{n\geq 0} u^*_n(r)x^n$ is rational. What other equivalence
relations on $\ca_n$ obtained by substitutions of words of equal
% 1/8/21 changed "are rational" to "yield rational gen fns"
length yield rational generating functions? For instance, the substitutions
$ab\leftrightarrow ba$ do not give rational generating functions for
$r\geq 2$. For $r=2$ the generating function is algebraic but not
rational, while for $r\geq 3$ it is D-finite but not algebraic
\cite[Exer.~6.3,\,6.54]{ec2}. Thus we can also ask in general when we
  get algebraic and D-finite generating functions.
%A further interesting example where rationality fails is the plactic
%monoid on $g\geq 3$ generators. When $g=3$ the substitutions (or
%relations) are $bac\leftrightarrow bca$ and $acb\leftrightarrow cab$.

\textsc{Note.} It is a nice exercise to show that if $f_1,f_2,\dots$
is a sequence of positive integers satisfying $f_1\neq f_2$ and
$f_{i+1}=f_i+f_{i-1}$ for all $i\geq 2$, then for all $n\geq 1$ the
sequence of nonzero coefficients of the polynomial $\prod_{i=1}^n
\left( 1+x^{f_i}\right)$ depends only on $n$.

\section{Generalizing the Fibonacci numbers} \label{sec:genfib}
What happens if we replace $F_{i+1}$ in the definition
\eqref{eq:indef} and its generalizations with some other sequence?  We
consider only sequences $f_1,f_2,\dots$ satisfying linear recurrences
with constant integer coefficients, called \emph{$C$-finite sequences}
by Zeilberger \cite{zeil}. Note that if $f_{i+1}\geq 2f_i$
for all $i$, then the nonzero coefficients of $\prod_{i=1}^n\left(
1+x^{f_i}\right)$ are all equal to 1, which is not so interesting. One
class of sequences that have more interesting behavior is given for
fixed $k\geq 1$ by 
  $$ F_{i+1}^{(k)}=F_{i}^{(k)}+F_{i-1}^{(k)}+\cdots+F_{i-k+1}^{(k)}, $$
say with initial conditions $F^{(k)}_{1}=F^{(k)}_{2}=\cdots
=F^{(k)}_{k}=1$. Thus $F_{i}^{(2)}=F_i$.

We conjecture that Conjecture~\ref{conj1} has a direct
$F^{(k)}$-analogue. 

\begin{conj}
Let $k\geq 2$, $h\geq 1$, $(a_1,\dots,a_h)\in\cc^h$, and
$P(x)\in\cc[x]$. Set  
  $$ I^{(k)}_{h,P,n}(x) = P(x)\prod_{i=1}^n\left(1+a_1x^{F^{(k)}_i}+
    a_2x^{F^{(k)}_{i+1}}+\cdots+a_hx^{F^{(k)}_{i+h-1}}\right). $$
Regarding $h,P,k$ as fixed, let $c_n(p)$ denote the coefficient
of $x^p$ in $I^{(k)}_{h,P,n}(x)$. For $\alpha=(\alpha_0,\dots,
\alpha_{m-1})\in\nn^m$ define
  $$ v^{(k)}_{h,P,\alpha}(n) = \sum_{p\geq 0}c_n(p)^{\alpha_0}
   c_n(p+1)^{\alpha_1}\cdots c_n(p+m-1)^{\alpha_{m-1}}. $$
Then the generating function $\sum_{n\geq
0}v^{(k)}_{h,P,\alpha}(n)x^n$ is rational. 
\end{conj}

For the special case
  $$ I_{H,P,n}^{(k)}(x) = \prod_{i=1}^n\left(
% 1/6/21 added ^{(k)} in next line
            1+tx^{F_{i+k-1}^{(k)}}\right), $$
we can prove this conjecture by a
combinatorial technique. When $k=2$ this gives a new proof of
Theorem~\ref{thm1}.

To give this proof, for $k\geq 2$ define $\cm^{(k)}$ to be the set of all
pairs $\pi$ of finite binary sequences of the same length, say $n$,
denoted 
   \beq \pi = \left(\begin{array}{cccc}
      a_1 & a_2 & \cdots & a_n\\
      b_1 & b_2 & \cdots  & b_n \end{array} \right), \label{eq:pi}
   \eeq 
such that
  $$ \sum_{i=1}^n a_iF^{(k)}_{i+k-1} = \sum_{i=1}^n
       b_iF_{i+k-1}^{(k)}. $$
It is easily seen that if $\pi\in\cm^{(k)}$ (where $\pi$ is given by
equation~\eqref{eq:pi}) and if
  $$ \sigma = \left(\begin{array}{cccc}
      c_1 & c_2 & \cdots & c_p \\
      d_1 & d_2 & \cdots  & d_p \end{array} \right) \in \cm^{(k)}, $$
then the concatenation
   $$ \pi\sigma = \left(\begin{array}{cccccccc}
      a_1 & a_2 & \cdots & a_n & c_1 & c_2 & \cdots & c_p\\
      b_1 & b_2 & \cdots  & b_n & d_1 & d_2 & \cdots &
      d_p \end{array} \right) $$
also belongs to $\cm^{(k)}$. Thus $\cm^{(k)}$ is a monoid under
concatenation. (The empty array is the identity element.)

For a binary letter $a=0,1$ let $a^j$ denote a sequence of $j$
$a$'s. For instance, $0^4=0000$.  Given $k\geq 2$, let $\cg^{(k)}$ be
the set of all pairs of binary sequences equal to
\beq \left( \begin{array}{c} 0\\0\end{array}\right)\ \mathrm{or}\
    \left( \begin{array}{c} 1\\1\end{array}\right),
      \label{eq:01} \eeq
or equal to one of the two forms (which differ by interchanging the
rows) 
\beq\label{eq:gens}
   \begin{aligned}
    \pi &= \left(\begin{array}{cccccccccccc} 1^k &
      * & 1^{k-1} & * & 1^{k-1} 
        & * & 1^{k-1} & * & \cdots & * & 1^{k-1} & 0\\ 
        0^k & * & 0^{k-1} & * & 0^{k-1} & *
        & 0^{k-1} & * & \cdots & * & 0^{k-1} & 1 \end{array}
    \right),\ \mathrm{or}\\[.5em]
    \sigma &= \left(\begin{array}{cccccccccccc}
      0^k & * & 0^{k-1} & * & 0^{k-1} & *
        & 0^{k-1} & * & \cdots & * & 0^{k-1} & 1\\
      1^k & * & 1^{k-1} & * & 1^{k-1}
        & * & 1^{k-1} & * & \cdots & * & 1^{k-1} & 0
         \end{array} \right), \end{aligned} \eeq
where * can be 0 or 1, but two *'s in the same column must be
equal. It's easy to see that $\cg^{(k)}\subset \cm^{(k)}$.
Then the following key lemma is fairly straightforward to prove.

\begin{lem} \label{lem:freegen}    
The set $\cg^{(k)}$ freely generates $\cm^{(k)}$. That is, every element $\pi$
of $\cm^{(k)}$ can be written uniquely as a product of words in $\cg^{(k)}$.
\end{lem}

% 1/9/21 deleted extra ) 
We can now state the main (nonconjectural) result of this section.

\begin{thm} \label{thm:vk2n}
Let $v^{(k)}_2(n,t)$ denote the sum of the squares of the coefficients
of the polynomial $\prod_{i=1}^n\left(
1+tx^{F_{i+k-1}^{(k)}}\right)$. Then
  $$ \sum_{n\geq 0} v^{(k)}_2(n,t)x^n = \frac{1-t^{k-1}(1+t^2)x^k}
    {1-(1+t^2)x-t^{k-1}(1+t^2)x^k+t^{k-1}(1+t^4)x^{k+1}}. $$
\end{thm}

\begin{proof}
Write $\ell(\pi)$ for the length of $\pi\in\cm^{(k)}$, and write
$N(\pi)$ for the total number of 1's in $\pi$.  Note that
$\ell(\pi\sigma) = \ell(\pi)+\ell(\sigma)$ and
$N(\pi\sigma) = N(\pi)+N(\sigma)$. Define
  $$ G^{(k)}(x) =\sum_{\pi \in\cg^{(k)}}
     t^{N(\pi)}x^{\ell(\pi)}. $$ 
By a standard simple argument (see \cite[{\S}4.7.4]{ec1}),
  $$ \sum_{n\geq 0} v^{(k)}_2(n,t)x^n = \frac{1}{1-G^{(k)}(x)}. $$
\indent We can use Lemma~\ref{lem:freegen} to compute $G^{(k)}(x)$.
The two generators in equation~\eqref{eq:01} contribute
$(1+t^2)x$ to 
$G^{(k)}(x)$. Now consider the generators $\pi$ and $\sigma$ of
equation~\eqref{eq:gens}. The two generators differ only by switching
rows, so consider just $\pi$. Suppose there are $j\geq 0$ columns of
$*$'s. The number of 1's in the remaining columns is
$k+j(k-1)+1$. The length of $\pi$ is $(j+1)k+1$. Since each of the $j$
% 1/6/20 added (x) in next line
columns of $*$'s has zero or two 1's, the contribution to $G^{(k)}(x)$
of all generators \eqref{eq:gens} is $t^{j(k-1)+k+1}(1+t^2)^j
x^{jk+k+1}$. The same is true of the second generator $\sigma$. Hence
  \beas G^{(k)}(x) & = & (1+t^2)x+2\sum_{j\geq 0} t^{j(k-1)+k+1}(1+t^2)^j
            x^{jk+k+1}\\  & = & (1+t^2)x+
            \frac{2t^{k+1}x^{k+1}}{1-t^{k-1}(1+t^2)x^k}.
  \eeas
It follows that
  \beas \sum_{n\geq 0} v_2^{(k)}(n,t)x^n & = & \frac{1}
        {1-(1+t^2)x-\frac{2t^{k+1}x^{k+1}}{1-t^{k-1}(1+t^2)x^k}}\\ & = &
        \frac{1-t^{k-1}(1+t^2)x^k}
    {1-(1+t^2)x-t^{k-1}(1+t^2)x^k+t^{k-1}(1+t^4)x^{k+1}}. \eeas
\end{proof}

Naturally we can ask how the statement and proof of
Theorem~\ref{thm:vk2n} can be extended. For any $r\geq 2$ let
$v^{(k)}_r(n,t)$ denote the sum of the $r$th powers of the coefficients
of the polynomial
$\prod_{i=1}^n\left(1+tx^{F_{i+k-1}}\right)$. Define the monoid
$\cm^{(k)}(r)$ analogously to $\cm^{(k)}$ by letting the elements of
$\cm^{(k)}(r)$ be $r$-tuples of binary words of the same length such
that $\sum_{i=1}^n a_iF^{(k)}_{i+k-1}$ is the same for all the $r$
words $a_1a_2\cdots a_n$. It's easy to see that $\cm^{(k)}(r)$ is a
free monoid, basically because if $\pi$ and $\sigma$ are $r$-tuples of
binary words such that $\pi\in\cm^{(k)}(r)$ and $\pi\sigma\in
\cm^{(k)}(r)$, then $\sigma\in\cm^{(k)}(r)$. However, finding the free
generators of $\cm^{(k)}(r)$ seems complicated for $r\geq 3$, and we
have not tried to do so. For $r=3$ we have the following conjecture,
due to Zeilberger \cite{zeil} (verified by him for $k\leq 5$), a
correction of the conjecture in the original version of the present paper.
We use the notation $u=t^{k-1}$ and
  $$ J_r^{(k)}(t,x) = \sum_{n\geq 0}v_r^{(k)}(n,t)x^n. $$

\begin{conj} \label{conj:v3k}
  We have
  $$ J_3^{(k)}(t,x) =\frac{1-u(u+1)(t^3+1)x^k+u^3(t^3-1)^2x^{2k}}
  {D_3^{(k)}(t,x)}, $$
where
% \beas D_3^{(k)}(t,x) & = & 1-x-u^3x^{2k+1}+u^3x^{2k}+ux^{k+1}-ux^k
%  +u^2x^{k+1}\\ & & -u^2x^k -t^3x
%  +t^3u^3x^{2k+1}-2t^3u^3x^{2k}-t^3ux^{k+1}\\ & &
%  -t^3ux^k-t^3u^2x^{k+1}-t^3u^2x^k\\ & & +t^6u^3x^{2k+1}+t^6u^3x^{2k}
%  +t^^u^2x^{k+1}+t^9u^3x^{2k+1}-t^9u^3x^{2k+1}. \eeas
%\end{conj}

\beas D_3^{(k)}(t,x) & = &
  1-(t^3+1)x-u(u+1)(t^3+1)x^k\\ & & +u(u+1)(t^6-t^3+1)x^{k+1}
  +u^3(t^3-1)^2x^{2k}\\ & &
  +u^3(t^3-1)(t^6-1)x^{2k+1}.
 \eeas
\end{conj}

Note that the numerator in the above conjecture consists of the terms
in the denominator with $x$-exponents $0,k,2k$. For 
higher values of $r$ the coefficients seem to be more complicated. For
instance, it seems that the coefficient of $x^4$ in the denominator of 
$J_4^{(k)}(t,x)$ is
$t^2(t^{12}+t^{10}-t^8-4t^6-t^4+t^2+1)$. The factor
$t^{12}+t^{10}-t^8-4t^6-t^4+t^2+1$  is
irreducible over $\qq$. We give below some conjectures when $t=1$.
%using the notation
%  $$ J_r^{(k)}(x) = \sum_{n\geq 0}v_r^{(k)}(n,1)x^n. $$

%Some experiments suggest that the sequences $f^{(k)}_{i}$ behave
%completely analogously to Fibonacci numbers with respect to the
%previous section. In fact, if we define $J_\alpha^{(k)}(x)$ to be the
%obvious $f_i^{(k)}$-analogue of $J_\alpha(x)$, then for fixed $\alpha$
%there seems to be a remarkable connection between the generating
%functions $J_\alpha^{(k)}(x)$. We can prove such a result when $\alpha=(2)$

%The following conjecture gives some
%examples of this behavior. It shows the inadequacy of our proof of
%Theorem~\ref{thm1}.

\begin{conj} \label{conj:jrkx}
We have
  \beas 
  J_4^{(k)}(1,x) & = &
  \frac{1-7x^k-2x^{2k}}{1-2x-7x^k-2x^{2k}+2x^{2k+1}}\\
    J_5^{(k)}(1,x) & = &
    \frac{1-11x^k-20x^{2k}}{1-2x-11x^k-8x^{k+1}-20x^{2k}+10x^{2k+1}}\\
  J_6^{(k)}(1,x) & = &
  \frac{1-17x^k-88x^{2k}-4x^{3k}}{D_6^{(k)}(1,x)}\\
    J_7^{(k)}(1,x) & = &
    \frac{1-26x^k-311x^{2k}-84x^{3k}}{D_7^{(k)}(1,x)},
   \eeas
where
   $$ D_6^{(k)}(1,x)=1-2x-17x^k-28x^{k+1}
   -88x^{2k}+26x^{2k+1}-4x^{3k}+4x^{3k+1} $$
and
  $$ D_7^{(k)}(1,x) = 1-2x-26x^k-74x^{k+1}-311x^{2k}+34x^{2k+1}-84x^{3k}
      +42x^{3k+1}.$$
\end{conj}

Theorem~\ref{thm1t} and Conjectures~\ref{conj:v3k} and \ref{conj:jrkx}
suggest the following conjecture. I am grateful to Doron Zeilberger
for pointing out that the original form of this conjecture was
incorrect. 

\begin{conj} \label{conj:drx}
For $r\geq 2$ there is an integer $m\geq 1$ (depending on $r$)
for which $J_r^{(k)}(t,x)$ has the form
  $$ J_r^{(k)}(t,x) =\frac{1+a_2(t,t^{k-1})x^k+a_4(t,t^{k-1})x^{2k}+\cdots
       +a_{2m}(t,t^{k-1})x^{mk}}{D_r^{(k)}(t,x)}, $$ 
where
   \beas D_r^{(k)}(t,x) & = &
  1+a_1(t,t^{k-1})x+a_2(t,t^{k-1})x^k+a_3(t,t^{k-1})x^{k+1}\\ & & +
  a_4(t,t^{k-1})x^{2k} +a_5(t,t^{k-1})x^{2k+1}\\ & &  
   + \cdots + a_{2m}(t,t^{k-1})x^{mk}+a_{2m+1}(t,t^{k-1})x^{mk+1},
   \eeas 
and where $a_i(t,u)$ is a polynomial in $t$ and $u$ (depending on $r$
but independent from $k$) such that $a_{2m+1}(t,t^{k-1})\neq 0$, and
possibly even $a_i(t,t^{k-1})\neq 0$ for all $0\leq i\leq
2m+1$. Moreover (in order to account for the odd denominator degrees
in equation~\eqref{eq:data}), the largest index $j$ for which
$a_j(1,1)\neq 0$ is odd. 
\end{conj}

%Note that this conjecture has the suprising consequence that once we
%know $J_r^{(2)}(t,x)$ (or even just its denominator),
%then we can immediately determine $J_r^{(k)}(t,x)$ for all $k$.

What other sequences satisfying linear recurrences with constant
coefficients have interesting behavior related to this paper?
%It seems (and shouldn't be difficult to prove) that altering the initial
%conditions for the recurrences $f_{i+1}=f_i+f_{i-1}+\cdots+f_{i-k+1}$
%does not change the denominators of the generating functions
%$J_\alpha^{(k)}(x)$. [check??] Moreover, the numerator also does not change when
%the initial conditions increase sufficiently quickly.
%More intriguing is the consideration of other recurrences.
We were unable to find any further recurrences with ``nice''
behavior. For instance, gfun fails to find rational generating
functions (using the values for $0\leq n\leq 40$) for the sum of the
squares of the coefficients of 
$\prod_{i=1}^n\left(1+x^{f_{i+2}}\right)$, when either $f_{i+1} =
f_i+f_{i-2}$ or $f_{i+1}=f_{i-1}+f_{i-2}$, with initial conditions
$f_1=f_2=f_3=1$. Zeilberger, however, is much more adept at
calculations, and he informs me (private communication, 24 March
2021), that for $f_{i+1}=f_i+f_{i-2}$, the sum of the squares of the
coefficients of $\prod_{i=1}^n (1+x^{f_{i+1}})$ has a generating
function $P(x)/Q(x)$ where $\deg P(x)=24$ and $\deg Q(x)=25$. For the
sum of the cubes, the denominator has degree 88. Similarly, for
$f_{i+1} = f_{i-1}+f_{i-2}$ the denominator of the generating function
for sum of the squares has degree 73. 
  
Note that in these two cases, the unique real zeros $\psi$ of the
corresponding characteristic polynomials $x^3-x^2-1$ and $x^3-x-1$ are
PV numbers, i.e., they are real algebraic integers greater than 1 all
of whose conjugates are less than 1 in absolute value. Similarly the
unique positive real zeros of $x^k-x^{k-1}-x^{k-2}-\cdots-x-1$, $k\geq
2$, are PV numbers. Thus Zeilberger \cite[p.~14]{zeil} conjectures
that his algorithms for computing $\sum u_\alpha(n)x^n$, correponding
to a $C$-finite sequence, terminates for all
$\alpha$ if and only if the the largest zero of the characteristic
polynomial of the recurrence is a PV number. We can make the somewhat
stronger conjecture that this condition on the recurrence is necessary
and sufficient for $\sum u_\alpha(n)x^n$ to be a rational function for
all $\alpha$. 

\section{Congruence properties}
For $0\leq a<m$, let $g_{m,a}(n)$ denote the number of coefficients of
$S_n(x)$ (defined by equation~\eqref{eq:stpr}) that are congruent to
$a$ modulo $m$. Reznick \cite{reznick} showed that the generating
function
  $$ G_{m,a}(x) = \sum_{n\geq 0} g_{m,a}(n)x^n $$
is rational. See also \cite[pp.~28--37]{rs:rut}, where some open
questions are on page 32. In particular, the denominator of
$G_{m,a}(x)$ has quite a bit of factorization that remains
unexplained. (For some small progress related to the denominator
factorization, see Bogdanov \cite{bog}.) For the proof that
$G_{m,a}(x)$ is rational, it is necessary to introduce auxiliary
generating functions $G_{m,a,b}(x)= \sum_{n\geq 0} g_{m,a,b}(n)x^n$,
where $g_{m,a,b}(n)$ is equal to the number of integers $0\leq k\leq
\deg S_n(x)$ for which $\anbc nk\equiv a\modd{m}$ and
$\anbc{n}{k+1}\equiv b \modd{m}$. 

We can do something analogous for the Fibonacci triangle. For $0\leq
a<m$, let $h_{m,a}(n)$ denote the number of coefficients of $I_n(x)$
(defined by equation~\eqref{eq:indef}) that are congruent to $a$ modulo
$m$. Define
  $$ H_{m,a}(x) = \sum_{n\geq 0} h_{m,a}(n)x^n. $$
The proof sketched in \cite{rs:rut} that $G_{m,a}(x)$ is rational
carries over, \emph{mutatis mutandis}, to $H_{m,a}(x)$. As in the
proof for $G_{m,a}(x)$, we need to introduce some auxiliary
generating functions that take into account consecutive coefficients 
of $I_n(x)$. However, we need also specify whether these coefficients
are the beginning, middle, or end of a string (as defined in
Section~\ref{sec:fp}). Thus we will have numbers like
$g^{(3,1)}_{m,a,b}(n)$ which count the number of integers $0\leq
k\leq \deg I_n(x)$ for which $\brbc nk$ ends a string and satisfies
$\brbc nk\equiv a\modd{m}$, while $\brbc{n}{k+1}$ begins a string and
satisfies $\brbc{n}{k+1}\equiv b\modd{m}$. When these procedures are
carried out we obtain the following result.

\begin{thm} \label{thm:hmax}
  The generating function $H_{m,a}(x)$ is rational.
\end{thm}

Naturally we would like to say more about $H_{m,a}(x)$ than just its
rationality. Here are some values suggested by gfun. None have been
proved. 

\beas H_{2,0}(x) & = & \frac{x^3(1-2x^2)}
      {(1-x)(1-x-x^2)(1-2x+2x^2-2x^3)}\\
      H_{2,1}(x) & = & \frac{1+2x^2}{1-2x+2x^2-2x^3}\\
      H_{3,0}(x) & = & \frac{2x^5(1-2x^2)}
       {(1-x)(1-x-x^2)(1-2x+2x^2-3x^3+4x^4-4x^5)}\\
       H_{3,1}(x) & = & \frac{1-2x+4x^2-6x^3+8x^4-10x^5+8x^6-6x^7}
          {(1-x)(1-x+x^2)(1-2x+2x^2-3x^3+4x^4-4x^5)}\\
      H_{3,2}(x) & = & \frac{x^3(1+2x^4)}
      {(1-x)(1-x+x^2)(1-2x+2x^2-3x^3+4x^4-4x^5)}\\
      H_{4,0}(x) & = & \frac{x^6(1-2x^2)(1-3x^2+4x^3-4x^4)}
      {(1-x)(1-x-x^2)(1-x^2+2x^4)(1-2x+2x^2-2x^3)^2}\\
        H_{4,1}(x) & = & \frac{1-2x+5x^2-8x^3+10x^4-12x^5+8x^6-6x^7}
        {(1-x)(1-2x+2x^2-2x^3)(1-x+2x^2-2x^3+2x^4)}\\
        H_{4,2}(x) & = & \frac{x^3(1+x^2)(1-2x^2)}
        {(1-x^2+2x^4)(1-2x+2x^2-2x^3)^2}\\
        H_{4,3}(x) & = & \frac{2x^5(1+x^2)}
         {(1-x)(1-2x+2x^2-2x^3)(1-x+2x^2-2x^3+2x^4)}
      \eeas
Note that just as for $G_{m,a}(x)$, there is a lot of denominator
factorization. Moreover, some of the numerators of $H_{m,a}(x)$ have
only two terms, in analogy to some numerators of $G_{m,a}(x)$ having
just one term.

We can try to extend Theorem~\ref{thm:hmax} to
  $$ I_n^{(k)}(x)=\prod_{i=1}^n \left( 1+x^{F_{i+k-1}^{(k)}}\right). $$
For $0\leq a<m$, let $h_{m,a}^{(k)}(n)$ denote the number of
coefficients of $I_n^{(k)}(x)$ that are congruent to $a$ modulo
$m$. Define
  $$ H_{m,a}^{(k)}(x) = \sum_{n\geq 0} h_{m,a}^{(k)}(n)x^n. $$

\begin{conj}
The generating function $H_{m,a}^{(k)}(x)$ is rational.
\end{conj}

We have some scanty evidence for a ``congruence analogue'' of
Conjecture~\ref{conj:drx}. For $(m,a)=(2,1)$ we found enough evidence
to conjecture the following.

\begin{conj}
We have
  $$ H_{2,1}^{(k)}(x) = \frac{1+2x^k}{1-2x+2x^k-2x^{k+1}}. $$
\end{conj}

For $(m,a)=(3,1)$ gfun suggests the following:

 \beas H_{3,1}^{(2)}(x) & = & \frac{1-2x+4x^2-6x^3+8x^4-10x^5+8x^6-6x^7}
       {1-4x+8x^2-12x^3+16x^4-20x^5+19x^6-12x^7+4x^8}\\[.5em]
   H_{3,1}^{(3)}(x) & = &
   \frac{1-2x+4x^3-6x^4+8x^6-10x^7+8x^9-6x^{10}}
        {D(x)}, \eeas
where
      $$ D(x) = 1-4x+4x^2+4x^3-12x^4+8x^5+8x^6-20x^7+11x^8 $$
      $$ \hspace{-2cm} +8x^9 -12x^{10}+4x^{11}. $$
The connection between the two numerators is obvious. Note the
denominator coefficients of $H_{3,1}^{(2)}(x)$ are obtained from those
of $H_{3,1}^{(3)}(x)$ by adding the coefficients of the pairs
$(x^2,x^3)$, $(x^5,x^6)$ and $(x^8,x^9)$, keeping the other
coefficients unchanged. It shouldn't be difficult to come up with
more general conjectures. Even better, of course, would be some
theorems! 

\textsc{Acknowledgment.} I am grateful to two anonymous referees for
helpful comments.

\end{document}